\documentclass[11pt]{article}
\usepackage{amsmath,amsfonts,latexsym,graphicx,amssymb,amsthm}
\usepackage[tips,matrix,arrow]{xy}
\usepackage{etex}
\usepackage{tikz}

\usepackage{xcolor}

\definecolor{violet}{cmyk}{1,0,0,0}
\definecolor{mygreen}{cmyk}{1,0,1,0}

\usepackage{hyperref}

\usepackage{setspace}

\newcommand{\C}{\mathbb{C}}

\newcommand{\R}{\mathbb{R}}
\newcommand{\Z}{\mathbb{Z}}
\newcommand{\fg}{\mathfrak{g}}

\newcommand{\Tr}{\mathrm{Tr}}

\newcommand{\Cox}[1]{h(\text{$#1$})}
\newcommand{\dCox}[1]{h^\vee(\text{$#1$})}

\newtheorem{thm}{Theorem}[section]
\newtheorem{lemma}[thm]{Lemma}
\newtheorem{defn}{Definition}[section]

\newtheorem{rem}[thm]{Remark}
\newtheorem{ack}{Acknowledgment}

\newcommand{\rad}{\mathrm{rad}}

\begin{document}
\title{Macdonald identities: revisited}
\author{K. Iohara and Y. Saito}

\maketitle

\begin{center}
\textit{To the memory of  Professor Masatoshi Noumi}
\end{center}

\begin{abstract}
In this note, after recalling a proof of the Macdonald identities for untwisted affine root systems, we derive 
the Macdonald identities for twisted affine root systems. 
\end{abstract}

\begin{center}
{\small \textbf{MSC2020}:  17B22 (primary) 11F20, 17B67 (secondary)
}
\end{center}

\pagestyle{plain}
\section{Introduction}
In 1972, I. G. Macdonald  \cite{Macdonald1972} obtained several identities involving 
products of the Dedekind $\eta$-function in relation with affine root systems via some 
complicated combinatorial method. Indeed, he first classified \textit{real} affine root systems
and tried finding a generalization of the denominator identity for affine root systems, that is,
an identity of the form $\text{infinite product}\; = \text{infinite sum}$, where the product is taken over 
roots in an affine root system and the sum is taken over the Weyl group of the same affine root system. 
He found that it requires some `\textit{extra factors}' in the product side. 
Later, V. Kac \cite{Kac1974}, R. Moody \cite{Moody1975} and H. Garland and 
J. Lepowsky \cite{GarlandLepowsky1976} independently showed that these `extra factors' 
can be interpreted as the factors related with the \textit{imaginary} roots. 
Starting from the denominator identity for each affine root system, via specialization, 
Macdonald obtained the so-called Macdonald identities. 
There are other known proofs of the Macdonald identities. For example, these identities, except for $BC_l^{(2)}$, were reproved in 1976 via some analysis on modular forms by Van Asch \cite{vAsch1976}, and H. D. Fegan \cite{Fegan1978} in 1978 reproved them only for untwisted case via some analysis of the heat equation on compact simple Lie groups. The latter work relies on the result of Kostant \cite{Kostant1976} where he rewrote the Macdonald identities in such a way as to involve the heat kernel on the corresponding compact Lie group. The reader may realize the richness of the subject. \\

In this note, after reviewing several facts, including the proof of Macdonald identities for
affine root systems of untwisted type, we present another proof of the Macdonald 
identity associated with a twisted affine root system, not of type $BC_l^{(2)}$, which uses unusual folding 
considered by K. Saito in \cite{Saito1985} and some facts from B. Kostant's beautiful 
paper \cite{Kostant1959}. For $BC_l^{(2)}$ case, we present one intermediate step and state the final result rapidly, 
as this case is too special and we believe that the reader who followed other cases may find no difficulty to fill the detail. 
Thus, most of this note is expository and our only modest contribution is 
\S \ref{sect:Macdonald2}. \\

Throughout this note, we don't fix any normalization of the metric, for which there is 
a price to pay in its presentation. \\

\medskip
\begin{ack}{\rm
Research of Y. S.  is supported by JSPS KAKENHI Grant number JP20K03568. We are grateful to an anonymous referee for his/her careful reading and suggestions. 
}\end{ack}
\section{Affine root systems}
Here, we recall some basic facts about finite and affine root systems, in particular, 
two types of folding procedure. 
We adopt the formulations in \cite{Saito1985} for later purpose. 
\subsection{Generalized root system}
Let $F$ be a real vector space and $I: F \times F \rightarrow \R$ be a symmetric bilinear 
form whose signature is $(l_+, l_0, l_-)$ for some non negative integers $l_+, l_0, l_- $ 
such that not all of them are zero. As usual, for any non-isotropic vector $\alpha \in F$, 
we set
\[ \alpha^\vee=\frac{2}{I(\alpha,\alpha)}\alpha \in F, \]
and define an isometry $s_\alpha\in \mathrm{O}(F,I)$ by
\[ s_\alpha(\lambda)=\lambda-I(\lambda, \alpha^\vee)\alpha. \]
\begin{defn} A non-empty discrete subset $R$ of $F$ is called a \textbf{generalized root 
system} if it satisfies
\begin{enumerate}
\item The lattice $Q(R)$ (called the \textbf{root lattice}), spanned by the elements of $R$, 
is full in $F$, i.e., $\R \otimes_{\Z} Q(R) \cong F$. 
\item For any $\alpha \in R$, one has $I(\alpha, \alpha) \neq 0$. 
\item For any $\alpha, \beta \in R$, one has $I(\alpha^\vee, \beta) \in \Z$.
\item For any $\alpha \in R$, one has $s_\alpha(R)=R$.
\item Assume that there exists two subsets $R_1, R_2$ of $R$ which are orthogonal and 
$R_1 \cup R_2=R$, then either $R_1$ or $R_2$ is empty. 
\end{enumerate}
\end{defn}
\vskip 0.1in
The subgroup of $\mathrm{O}(F,I)$ generated by the reflections 
$s_\alpha$ ($\alpha \in R$) is called the \textbf{Weyl group} of $R$ and is denoted by $W(R)$.

A finite root system $R_f$ is a generalized root system belonging to $F_f$ with a positive definite metric $I_f$, and an \textbf{affine root system} $R_{af}$ is a generalized root system belonging to $F$, containing $F_f$ as a subspace of codimension $1$, equipped with a positive semidefinite symmetric bilinear form $I$ whose restriction to $F_f \times F_f$ is $I_f$. The image of $R_{af}$ via the canonical projection $F \twoheadrightarrow F/\rad(I) \cong F_f$, denoted by $R_{af}/\rad(I)$, is called a quotient root system which can be identified with a finite root system $R_f$ belonging to $(F_f,I_f)$. 

For a finite root system $R_f$ of reduced type, let $(R_f)_s$ and $(R_f)_l$ be the set of short and long roots, respectively. Here and after, we fix a set of simple roots $\Pi_f=\{\alpha_i\}_{1\leq i\leq l}$ of $R_f$, thus the set $R_f^+$ of positive roots of $R_f$. We also set $(R_f^+)_\ast=(R_f)_\ast \cap R_f^+$ ($\ast \in \{ \, s, l\, \}$). \\
For $R_f$ of type $BC_l$, the set of roots of middle length is denoted by $(R_f)_m$. We fix a set of simple roots $\Pi_f=\{\alpha_i\}_{1\leq i\leq l}$ of $R_f$, hence of $(R_f)_s \cup (R_f)_m$ which is of type $B_l$, thus the set $R_f^+$ of positive roots of $R_f$. We assume that $\alpha_l \in (R_f)_s$.

\subsection{Affine root systems}\label{sect:aff-root-sys}
Let $R_{af}$ be a reduced affine root system belonging to $(F,I)$, and $R_f$ be the finite root system of type $X_l$ belonging to $(F_f,I_f)$ which is isomorphic to the quotient $R_{af}/\rad(I)$. \\

Let $\delta \in \rad(I)$ be a generator of the lattice $\rad(I) \cap Q(R_{af})$. 
It is known (cf. \cite{Kac1990} and/or \cite{MoodyPianzola1995}) that if 
\begin{enumerate}
\item $R_f$ of type $A_l$ ($l\geq 1$), $D_l$ ($l\geq 4$) or $E_6, E_7$ or $E_8$,  $R_{af}$ is of type $X_l^{(1)}$ defined by
\[ R_{af}=R_f+\Z \delta, \]
if
\item $R_f$ is of type $B_l$, $C_l$, $F_4$ or $G_2$,  there exists $t \in \{1,2\}$ for $R_f$ of type $B_l, C_l$ or $F_4$ and $t \in \{1,3\}$ for $R_f$ of type $G_2$ such that $R_{af}$ is of type $X_l^{(t)}$ defined by
\[ R_{af}=\big( (R_f)_s+\Z \delta \big) \cup \big( (R_f)_l+t\Z \delta\big), \]
otherwise
\item $R_f$ is of type $BC_l$ and $R_{af}$ is of type $BC_l^{(2)}$ is defined by 
\[ R_{af}=\big( (R_f)_s+\Z \delta\big) \cup \big((R_f)_m+\Z \delta\big) \cup \big((R_f)_l+(1+2\Z)\delta\big). \]
\end{enumerate}
Affine root system of type $X_l^{(t)}$ is called \textbf{untwisted} type if $t=1$ and \textbf{twisted} type otherwise. The number $t$ is called 
the \textbf{tier number}. \\

Here and after, we use the nomenclature for the reduced affine root systems 
due to K. Saito \cite{Saito1985}, since they reflect the structure of root systems 
themselves more than those of V. Kac \cite{Kac1990}. 
\begin{rem}
There are several other
known nomenclatures for reduced affine root systems. For the sake of reader's 
convenience, we clarify the relationships between them in the following table$:$ \\
\begin{center}
\begin{tabular}{|c||c|c|c|c|c||c|} \hline
K. Saito $($\cite{Saito1985}$)$ & $X_l^{(1)}$  & $B_l^{(2)}$ & $C_l^{(2)}$ 
& $F_4^{(2)}$ & $G_2^{(3)}$ & $BC_l^{(2)}$\\ \hline
\hline
V. Kac $($\cite{Kac1990}$)$ & $X_l^{(1)}$ & $\begin{matrix} {\tiny D_{l+1}^{(2)}} \\ {\tiny l \geq 2} \end{matrix}$  & $\begin{matrix} {\tiny A_{2l-1}^{(2)}} \\ {\tiny l\geq 3} \end{matrix}$ & 
 $E_6^{(2)}$ & $D_4^{(3)}$  & $A_{2l}^{(2)}$\\ \hline
R. Moody $($\cite{Moody1969}$)$& $X_{l,1}$ & $B_{l,2}$ & $C_{l,2}$  
& $F_{4,2}$ & $G_{2,3}$ & 
$\begin{matrix}
{\tiny A_{1,2}} & \text{\tiny $l=1$}\\ 
{\tiny BC_{l,2}} & \text{\tiny $l>1$} 
\end{matrix}$
\\ \hline
I. Macdonald $($\cite{Macdonald1972}$)$ & 
$\begin{matrix}
{\tiny X_l=X_l^\vee} & \text{\tiny $X$: of type $ADE$}\\ 
{\tiny X_l} &  \text{\tiny $X$: of type $BCFG$} 
\end{matrix}$
& $C_{l}^\vee$ & $B_{l}^\vee$  
& $F_{4}^\vee$ & $G_{2}^\vee$ & $BC_l=BC_l^\vee$\\ \hline
R. Carter $($\cite{Carter2005}$)$ & $\widetilde{X}_l$
& $\widetilde{C}_l^{\mathrm{t}}$ & $\widetilde{B}_l^{\mathrm{t}}$ 
& $\widetilde{F}_4^{\mathrm{t}}$ & $\widetilde{G}_2^{\mathrm{t}}$
& $\widetilde{C}_l'$ \\ \hline
\end{tabular}
\end{center} 
\end{rem}

In the rest of this subsection, we assume that affine root system $R_{af}$ is of type $X_l^{(t)}$. \\

Set
\[ \alpha_0=\begin{cases}  \; \delta-\theta \; & t=1, \\ 
                                          \; \delta-\theta_s\; &t\neq 1 \; \text{and not of type $BC_l^{(2)}$}, \\
                                          \; \delta -2\theta_s \; & \text{$R_{af}$ of type $BC_l^{(2)}$}, 
                   \end{cases} \]
where $\theta \in R_f$ and $\theta_s \in (R_f)_s$ are highest roots. Then, the set
\[ \Pi_{af}:=\{ \alpha_0\} \cup \Pi_f \]
is a set of simple roots, and the matrix $A=\big(\, I(\alpha_i^\vee, \alpha_j)\,\big)_{0\leq i,j \leq l}$ is called the \textbf{generalized Cartan matrix}. \\

Let $a_i, \; a_i^\vee$ ($0\leq i\leq l$) be the positive integers such that 
\begin{enumerate}
\item the integers $a_0, a_1, \ldots, a_l$ (resp. $a_0^\vee, a_1^\vee, \ldots, a_l^\vee$), called \textbf{labels} (resp \textbf{colabels}), are coprime to each other and
\item they satisfy
\[  A\begin{pmatrix} a_0 \\ a_1 \\ \vdots \\ a_l \end{pmatrix}=  \begin{pmatrix} 0 \\ 0 \\ \vdots \\ 0 \end{pmatrix} \qquad \text{and} \qquad 
   \begin{pmatrix} a_0^\vee & a_1^\vee & \cdots & a_l^\vee \end{pmatrix} A=\begin{pmatrix} 0 & 0 & \cdots & 0 \end{pmatrix}. \]
\end{enumerate}
The numbers $\Cox{X_l^{(t)}}:=\sum_{i=0}^l a_i$ and $\dCox{X_l^{(t)}}:=\sum_{i=1}^l a_i^\vee$ are called the \textbf{Coxeter} and \textbf{dual Coxeter} number, respectively. 
\begin{rem}
\begin{enumerate}
\item For $R_{af}$, not of type $BC_l^{(2)}$,
\begin{enumerate}
\item[(i)] $a_0=a_0^\vee=1$ and the vector $\sum_{i=0}^l a_i \alpha_i$ is equal to $\theta$ if $t=1$ and to $\theta_s$ otherwise. 
\item[(ii)] For $t=1$, the number $\Cox{X_l^{(1)}}$ coincides with the \textbf{Coxeter number} $\Cox{X_l}$ of the finite root system of type $X_l$, i.e., the order of a Coxeter element.  
\end{enumerate}
\item For $R_{af}$ of type $BC_l^{(2)}$, 
\begin{enumerate}
\item[(i)] $a_0=1$, $a_l=2$ and $a_0^\vee=2$, $a_l^{\vee}=1$. 
\end{enumerate}
\end{enumerate}
\end{rem}

\subsection{Folding and Mean-Folding}\label{sect:folding}
Here we recall two types of folding procedures introduced by K. Saito \cite{Saito1985}. 
Let $R$ be a finite or an affine reduced root system belonging to the pair $(F,I)$ of a real vector space with a symmetric bilinear form $I$, and let
$\Gamma_R$ be its Dynkin diagram, whose nodes are identified with simple roots. Assume that the oriented graph $\Gamma_R$ admits a non-trivial finite order automorphism 
$\sigma$ which fixes at least one node of $\Gamma_R$. Remark that the map $\sigma$ is a restriction of an isometry of $(F,I)$, which we denote by the same symbol $\sigma$. Let $\langle \sigma \rangle$ be the group generated by $\sigma$. Let $\Tr^{\langle \sigma \rangle}$ and $\Tr_{\langle \sigma \rangle}$ be the linear maps from $F$ to its fixed point subspace $F^{\langle \sigma \rangle}$ defined by
\[ \Tr^{\langle \sigma \rangle}(x)=\sum_{h \in \langle \sigma \rangle} h.x, \qquad \Tr_{\langle \sigma \rangle}(x)=\frac{1}{\vert \langle \sigma \rangle\vert}\sum_{h \in \langle \sigma \rangle} h.x. \]
The sets $\Tr^{\langle \sigma \rangle}(R)$ and $\Tr_{\langle \sigma \rangle}(R)$ are again finite or affine root systems, called the \textbf{folding} and \textbf{mean-folding} of $R$, respectively.  Their Dynkin diagrams can be obtained by ``folding'' appropriately, with respect to $\sigma$, the Dynkin diagram of $\Gamma_R$. 
\begin{rem}
\begin{enumerate}
\item Usually, the folding of $R$ is defined as the image of $R$ in the \textbf{$\langle \sigma \rangle$-coinvariant} of $F$, and such a operation corresponds to the mean-folding $\Tr_{\langle \sigma \rangle}$. 
\item  This automorphism $\sigma$ is a restriction of an isometry of $F$, which also acts 
on the dual root system $R^\vee$ and induces an action on its Dynkin diagram $\Gamma_{R^\vee}$ denoted by $\sigma^\vee$.
Two foldings are related as follows:
\[ \Tr^{{\langle \sigma^\vee \rangle}}(R^{\vee})=\Big(\Tr_{{\langle \sigma \rangle}}(R)\Big)^\vee \qquad 
   \Tr_{\langle \sigma^\vee \rangle}(R^\vee)=\Big(\Tr^{\langle \sigma \rangle}(R)\Big)^\vee. \]
\end{enumerate}
\end{rem}

When a non-trivial diagram automorphism $\sigma_{af}$ of an affine root system $R_{af}$ can be obtained by extending trivially a diagram automorphism $\sigma_f$ of the finite root system $R_f \cong R_{af}/\rad(I)$, namely, 
\[ \sigma_{af}(\alpha_i):=\begin{cases}\;  \alpha_i \; & i=0, \\
                                                           \;    \sigma_f(\alpha) \; & 1\leq i\leq l, \end{cases}
\]
the (mean-)folding of $(R_f, \sigma_f)$ and $(R_{af},\sigma_{af})$ are given in the following table:
 set $H_f=\langle \sigma_{f}\rangle$ and $H_{af}=\langle \sigma_{af}\rangle$.  \\
 \vskip 0.2In 
 
\begin{table}[h] \caption{(Mean-)Folding} \label{table:folding}
\begin{center}
\begin{tabular}{|c||c|c|c|c|} \hline 
Type of $R_f$ & $A_{2l-1}$ & $D_{l+1}$ & $E_6$ & $D_4$  \\ \hline \hline

\begin{tikzpicture}[scale=0.5]
\draw (0,0.2) node {$\sigma_f$};
\draw (0,0) node {$\phantom{i}$};
\draw (0,-0.6) node[below] {$\phantom{a}$}; 
\end{tikzpicture}
& 

\begin{tikzpicture}[scale=0.8]
\draw (0,-0.5) circle (0.1); \draw (0.1,-0.5) -- (0.9,-0.5); \draw (1,-0.5) circle (0.1); 
\draw (1.1, -0.5) -- (1.5, -0.5); \draw [dashed] (1.5,-0.5) -- (2.5,-0.5); \draw (2.5,-0.5) -- (2.9,-0.5); 
\draw (3,-0.5) circle (0.1); 
\draw (3.07,-0.43) -- (3.43, -0.07); 
\draw (3.5,0) circle (0.1); 
\draw (3.43,0.07) -- (3.07,0.43);
\draw (0,0.5) circle (0.1); \draw (0.1,0.5) -- (0.9,0.5); \draw (1,0.5) circle (0.1); 
\draw (1.1, 0.5) -- (1.5, 0.5); \draw [dashed] (1.5,0.5) -- (2.5,0.5); \draw (2.5,0.5) -- (2.9,0.5); 
\draw (3,0.5) circle (0.1); 
\draw [dotted, thick, <->] (0,-0.3) -- (0,0.3);
\draw [dotted, thick, <->] (1,-0.3) -- (1,0.3); 
\draw [dashed] (1.5,0) -- (2.5,0);
\draw [dotted, thick, <->] (3,-0.3) -- (3,0.3);
\draw (2,-0.6) node[below] {$\phantom{a}$}; 
\end{tikzpicture}

& 

\begin{tikzpicture}[scale=0.8]
\draw (0,0) circle (0.1); \draw (0.1,0) -- (0.9,0); \draw (1,0) circle (0.1); 
\draw (1.1, 0) -- (1.5, 0); \draw [dashed] (1.5,0) -- (2.5,0); \draw (2.5,0) -- (2.9,0); 
\draw (3,0) circle (0.1); 
\draw (3.07,0.07) -- (3.43, 0.43); 
\draw (3.5,0.5) circle (0.1); 
\draw (3.07,-0.07) -- (3.47,-0.43);
\draw (3.5,-0.5) circle (0.1); 
\draw [dotted, thick, <->] (3.7,-0.35) ..controls +(0.2,0.2) and +(0.2,-0.2).. (3.7,0.35);
\draw (2,-0.6) node[below] {$\phantom{a}$}; 
\end{tikzpicture}

& 

\begin{tikzpicture}[scale=0.8]
\draw (0,0) circle (0.1); \draw (0.1,0) -- (0.9,0);\draw (1,0) circle (0.1); 
\draw (1.09,0.045) -- (1.91,0.455); \draw (2,0.5) circle (0.1);
\draw (1.09,-0.045) -- (1.91,-0.455); \draw (2,-0.5) circle (0.1);
\draw (2.1,0.5) -- (2.9,0.5); \draw (3,0.5) circle (0.1);
\draw (2.1,-0.5) -- (2.9,-0.5); \draw (3,-0.5) circle (0.1);
\draw [dotted, thick, <->] (2,-0.3) -- (2,0.3);
\draw [dotted, thick, <->] (3,-0.3) -- (3,0.3);`
\draw (2,-0.6) node[below] {$\phantom{a}$}; 
\end{tikzpicture}

& 

\begin{tikzpicture}[scale=0.8]
\draw (0,0) circle (0.1); 
\draw (1.5,0) circle (0.1); \draw (0.1,0) -- (1.4,0);
\draw (1.5,1) circle (0.1); \draw (0.083,0.055) -- (1.417,0.945);
\draw (1.5,-1) circle (0.1); \draw (0.083,-0.055) -- (1.417,-0.945);
\draw [dotted, thick, ->] (1.7,-0.85) ..controls +(0.3,0.3) and +(0.3,-0.3).. (1.7,0.85);
\draw [dotted, thick, ->] (1.3,0.75) ..controls +(-0.2,-0.2) and +(-0.2,0.2).. (1.4,0.15);
\draw [dotted, thick, ->] (1.3,-0.15) ..controls +(-0.2,-0.2) and +(-0.2,0.2).. (1.4,-0.75);
\draw (1,1.2) node {$\phantom{a}$};
\end{tikzpicture}

  \\ \hline \hline
Type of $\Tr^{H_f}(R_f)$ & $B_l$ & $C_l$ & $F_4$ & $G_2$ \\ \hline
Type of $\Tr_{H_f}(R_f)$ & $C_l$ & $B_l$ & $F_4$ & $G_2$ \\ \hline \hline

Type of $\Tr^{H_{af}}(R_{af})$ & $B_l^{(2)}$ & $C_l^{(2)}$ & $F_4^{(2)}$ & $G_2^{(3)}$    \\ \hline
Type of $\Tr_{H_{af}}(R_{af})$ & $C_l^{(1)}$ & $B_l^{(1)}$ & $F_4^{(1)}$ & $G_2^{(1)}$  \\ \hline
\end{tabular}
\end{center}
\end{table}

We will use the fact that any twisted affine root system, except for $BC_l^{(2)}$, can be obtained as a folding of an $ADE$-type affine root system.
\begin{rem} 
Indeed, one can realize the root system $R_{af}$ of type $BC_l^{(2)}$ as follows. Let $\sigma_{af}$ be a diagram automorphism $D_{2l+2}^{(1)}$ of order $4$. Then, 
\[  \Tr^{\langle \sigma_{af}\rangle}(R(D_{2l+2}^{(1)})) \cong \Tr_{\langle \sigma_{af}\rangle}(R(D_{2l+2}^{(1)})) \cong R(BC_l^{(2)}).
\]
This automorphism $\sigma_{af}$, admitting exactly one fixed point, permutes the $4$ extremal nodes and hence is not of type described above. 
\end{rem}


\subsection{Affine Weyl groups} \label{sect:Aff-Weyl}

Here we extend the pair $(F,I)$ considered in \S \ref{sect:aff-root-sys} to the real vector space $\widehat{F}$ by adding the dual of $\rad(I)$; $\widehat{F}:=F \oplus \rad(I)^\ast$ and a symmetric bilinear form $\hat{I}$ such that $\big(\rad(I)^\ast \big)^{\perp}=F_f \oplus \rad(I)^\ast$ and its restriction to $\rad(I) \times \rad(I)^\ast$ is the dual pairing. A node $\alpha_i$ of the Dynkin diagram of $R_{af}$ is called \textbf{special index} if $\delta-a_i\alpha_i \in R_{af}$.  \\

Let $\widehat{W}(R_{af})$ be the subgroup of $\mathrm{O}(\widehat{F},\hat{I})$ generated by reflections $\hat{s}_\alpha$ ($\alpha \in R_{af} \subset F \subset \widehat{F}$) defined by
\[ \hat{s}_\alpha(\lambda)=\lambda - \hat{I}(\lambda, \alpha^\vee)\alpha \qquad \lambda \in \widehat{F}. 
\]
It is clear that the finite subgroup generated by $\hat{s}_\alpha$ ($\alpha \in R_f$) is isomorphic to the finite Weyl group $W(R_f)$. 
For a non-isotropic vector $\gamma \in F\subset \widehat{F}$, define $t_\gamma \in \mathrm{O}(\widehat{F},\hat{I})$ (cf. \cite[(6.5.2)]{Kac1990}) by
\[ t_\gamma(\lambda)=\lambda+\hat{I}(\lambda, \delta)\gamma - \Big( \hat{I}(\lambda, \gamma)+\frac{1}{2}\hat{I}(\gamma, \gamma)\hat{I}(\lambda,\delta)\Big)\delta.
\]
By direct computation, it can be verified that 
\begin{enumerate}
\item $w t_\gamma w^{-1} =t_{w(\gamma)}$ for any $w \in W(R_f)$,  and 
\item $t_{\gamma_1}t_{\gamma_2}=t_{\gamma_1+\gamma_2}$ for any $\gamma_1, \gamma_2 \in F$. 
\end{enumerate}
First, we consider the affine root systems $R_{af}$, where $\alpha_0$ is a speical index, i.e., $R_{af}$ is not of type $BC_l^{(2)}$. In this case, we describe the group $\widehat{W}(R_{af})$ under the splitting $F=F_f \oplus \rad(I)$. 

For an untwisted $R_{af}$, since $\hat{s}_{\alpha_0}\hat{s}_{\delta-\alpha_0}=t_{\theta^\vee}$ and $\theta \in R_f$ is a long root, it follows that $\theta^\vee \in (R_f^\vee)_s$ and its $W(R_f)$-orbit generates the coroot lattice $Q(R_f^\vee)$. Thus we have
\begin{equation}\label{Weyl-untwisted}
\widehat{W}(R_{af}) \cong W(R_f) \ltimes Q(R_f^\vee). 
\end{equation}
Likewise, for a twisted $R_{af}$, not of type $BC_l^{(2)}$, as $\hat{s}_{\alpha_0}\hat{s}_{\delta-\alpha_0}=t_{\theta_s^\vee}$ and $\theta_s\in R_f$ is a short root, it follows that the $W(R_f)$-orbit of $\theta_s^\vee=\dfrac{2}{\hat{I}(\theta_s,\theta_s)}\theta_s$ generate the lattice $\dfrac{2}{\hat{I}(\theta_s,\theta_s)}Q(R_f)$, i.e., 
\begin{equation}\label{Weyl-twisted}
\widehat{W}(R_{af}) \cong W(R_f) \ltimes \Bigg(\frac{2}{\hat{I}(\theta_s,\theta_s)}Q(R_f)\Bigg). 
\end{equation}

Second, we consider the affine root systems $R_{af}$ of type $BC_l^{(2)}$. In this case, as the unique simple short root $\alpha_l$ is the only special index,  we shall set $\Pi_f'=\{\alpha_i\}_{0\leq i\leq l-1}$, $R_f'=R_{af} \cap \Z \Pi'_f\cong R(C_l)$ and $F_f'=\bigoplus_{i=0}^{l-1} \R \alpha_i$, and describe the group $\widehat{W}(R_{af})$ under the splitting $F=F_f' \oplus \rad(I)$.   \\
As $\hat{s}_{\alpha_l}\hat{s}_{\delta-2\alpha_l}=t_{\theta^\vee}$, where $\theta \in R_f'$ is a highest root hence is a long root, it follows that $\theta^\vee \in (R_f')^\vee_s$ and its $W(R_f')$-orbit generates the coroot lattice $Q((R_f')^\vee)$. Thus, similarly to \eqref{Weyl-untwisted},  we have 
\begin{equation}\label{Weyl-BC}
\widehat{W}(R_{af}) \cong W(R_f') \ltimes Q\big((R_f')^\vee)\big). 
\end{equation} 

\begin{rem} It follows from the above isomorphisms that
the Weyl group of $R_{af}$, which is a subgroup of $\mathrm{O}(F,I)$, is isomorphic to $\widehat{W}(R_{af})$.
\end{rem}

Let $M$ be the lattice in the right hand side of above isomorphism $\widehat{W}(R_{af}) \cong W_f \ltimes M$.  The above isomorphism is clearly given by associating $(u, \gamma) \in W_f \times M$ with $ut_\gamma \in \widehat{W}(R_{af})$. It is clear that 
\[ \det{}_{\widehat{F}}(ut_\gamma)=\det{}_{\widehat{F}}(u)=\det{}_{F_f}(u), \]
that is, the determinant of an element $ut_\gamma$ of $\widehat{W}(R_{af})$ depends only  on its finite part.\\

\section{The Macdonald identities}
In this section, after recalling the denominator identity for finite and affine root systems and the Macdonald identity for untwisted affine root systems, we prove the Macdonald identity for twisted affine root systems, which is not of type $BC_l^{(2)}$, in a unified manner. The $BC_l^{(2)}$ case will be treated separately. We often use the symbols introduced in the previous section without reintroducing them. 
\subsection{Denominator identity}
Here, we recall the denominator identity for finite and affine root systems, separately. \\

First we recall the so-called \textbf{Weyl vector}. Let $R_f$ be a reduced finite root system belonging to $(F_f,I_f)$, $\Pi_f=\{\alpha_i\}_{1\leq i \leq l} \subset R_f$ the set of simple roots of $R_f$ and $R_f^+$ the set of positive roots. Define the Weyl vector $\rho_f \in F_f$ by
\[
\rho_f=\frac{1}{2}\sum_{\alpha \in R_f^+} \alpha. \]
It can be checked that $I_f(\rho_f,\alpha_i^\vee)=1$ for any $1\leq i\leq l$.  
For an affine root system $R_{af}$ belonging to $(F,I)$ whose quotient root system $R_{af}/\rad(I)$ is isomorphic to $R_f$ (not of type $BC_l$), let $\Pi_{af}$ be the set of simple roots recalled in \S \ref{sect:aff-root-sys}. Regard $R_{af}$ as a subset of $\widehat{F}$. For $0\leq i \leq l$, let $\Lambda_i \in F_f \oplus \rad(I)^\ast$ be the vector satisfying 
$\hat{I}(\Lambda_i,\alpha_j^\vee)=\delta_{i,j}$. The Weyl vector $\rho \in \widehat{F}$ is defined by
\[ \rho=\sum_{i=0}^l \Lambda_i. \]
It appears that $\rho-\rho_f \in \rad(I)^\ast$. Notice that this construction depends on the splitting $F=F_f\oplus \rad(I)$. \\
For the case when $R_{af}$ is of type $BC_l^{(2)}$, we use another splitting $F=F_f' \oplus \rad(I)$ defined in \S \ref{sect:Aff-Weyl}.
Let $\Pi_{af}$ be the set of simple roots recalled in \S \ref{sect:aff-root-sys}.  
For $0\leq i \leq l$, let $\Lambda_i' \in F_f' \oplus \rad(I)^\ast$ be the vector satisfying 
$\hat{I}(\Lambda_i',\alpha_j^\vee)=\delta_{i,j}$. The Weyl vector $\rho \in \widehat{F}$ is defined by
\[ \rho=\sum_{i=0}^l \Lambda_i'. \]
Recall the subset $\Pi_f'$ of $\Pi_{af}$ defined in  \S \ref{sect:Aff-Weyl}. The Weyl vector $\rho_f'$ of $R_f'$ satisfies $\rho-\rho_f' \in \rad(I)^\ast$. 
\\

Second, we recall a certain algebra $\mathcal{E}$ over $\C$. For $\lambda \in \widehat{F}$, set 
\[ D(\lambda)=\lambda-\Z_{\geq 0}\Pi_f=\big\{ \, \mu \in \widehat{F}\, \big\vert \, \exists\, m_i \in \Z_{\geq 0}\; \text{s.t.}\; \lambda-\mu=\sum_{i=0}^l m_i\alpha_i. \, \big\}. \]
The algebra $\mathcal{E}$ consists of the elements of the form $\sum_{\lambda \in \widehat{F}} c_\lambda e^\lambda$ such that there exists a finite number of vectors $\mu_1, \mu_2, \cdots, \mu_r \in \widehat{F}$ satisfying
\[ \big\{ \,\lambda \in \widehat{F}\, \big\vert \, c_\lambda \neq 0\, \big\} \subset \bigcup_{i=1}^r D(\mu_i). \]
The algebra structure on $\mathcal{E}$ is defined by $e^{\lambda} \cdot e^{\mu}=e^{\lambda+\mu}$.  \\

Now, we recall the so-called denominator identity.  The denominator identity of a finite root system $R_f$ is
\begin{equation}\label{eq:denominator-id-fin}
e^{\rho_f}\prod_{\alpha \in R_f^+}\big(1-e^{-\alpha}\big)=\sum_{u \in W(R_f)} \det{}_{F_f}(u)e^{u(\rho_f)}. 
\end{equation}
Notice that the left hand side can be rewritten as follows:  
\[ e^{\rho_f}\prod_{\alpha \in R_f^+}(1-e^{-\alpha})=\prod_{\alpha \in R_f^+}\Big( e^{\frac{1}{2}\alpha}-e^{-\frac{1}{2}\alpha}\Big). \]
For an affine root system $R_{af}$ of untwisted type, it is
\begin{equation}\label{eq:denominator-id-affine-I}
\begin{split}
&e^{\rho-\frac{\hat{I}(\rho,\rho)}{2\hat{I}(\rho,\delta)}\delta}\prod_{\alpha \in R_f^+}\big(1-e^{-\alpha}\big)
\prod_{n=1}^\infty \bigg( \big(1-e^{-n\delta}\big)^l \prod_{\alpha \in R_f}\big(1-e^{-\alpha-n\delta}\big) \bigg)\\
=
&\sum_{w \in \widehat{W}(R_{af})} \det{}_{\widehat{F}}(w)e^{w(\rho)-\frac{\hat{I}(\rho,\rho)}{2\hat{I}(\rho,\delta)}\delta}.
\end{split}
\end{equation}
For a twisted affine root system $R_{af}$ not of type $BC_l^{(2)}$, 
\begin{equation}\label{eq:denominator-id-affine-II}
\begin{split}
&e^{\rho-\frac{\hat{I}(\rho,\rho)}{2\hat{I}(\rho,\delta)}\delta}\prod_{\alpha \in R_f^+}\big(1-e^{-\alpha}\big)
\prod_{n=1}^\infty \bigg( \big(1-e^{-n\delta}\big)^{\vert \Pi_{f,s} \vert} \big(1-e^{-tn\delta}\big)^{\vert \Pi_{f,l} \vert} \\
& \phantom{e^{\rho}\prod_{\alpha \in R_f^+}\big(1-e^{-\alpha}\big)
\prod_{n=1}^\infty \bigg( }
\times \prod_{\alpha \in (R_f)_s}\big(1-e^{-\alpha-n\delta}\big) \prod_{\alpha \in (R_f)_l}\big(1-e^{-\alpha-tn\delta}\big) \bigg)\\
=
&\sum_{w \in \widehat{W}(R_{af})} \det{}_{\widehat{F}}(w)e^{w(\rho)-\frac{\hat{I}(\rho,\rho)}{2\hat{I}(\rho,\delta)}\delta},
\end{split}
\end{equation}
where we set $\Pi_{f,s}=\Pi_f \cap (R_f)_s$ and $\Pi_{f,l}=\Pi_f \cap (R_f)_l$. Here, the right hand side is a $\widehat{W}(R_{af})$ alternating sum of $e^{\rho-\frac{\hat{I}(\rho,\rho)}{2\hat{I}(\rho,\delta)}\delta}$, instead of $e^{\rho}$. This shift is made so that the vector $\rho':=\rho-\dfrac{\hat{I}(\rho,\rho)}{2\hat{I}(\rho,\delta)}\delta$ satisfy $\hat{I}(\rho',\rho')=0$. In this way, one may expect that the right hand side of \eqref{eq:denominator-id-affine-I} and \eqref{eq:denominator-id-affine-II} can be expressed in terms of an alternating sum involving $\theta$ constants which is the case as we will see in the next subsections.\\

\subsection{Macdonald identity I: untwisted case}\label{sect:Macdonald1}
Here, we recall how one can derive the Macdonald identity for an untwisted affine root system $R_{af}$ of type $X_l^{(1)}$. \\

First, we divide the both side of \eqref{eq:denominator-id-affine-I} by $e^{\rho}\prod_{\alpha \in R_f^+}(1-e^{-\alpha})$. Its left hand side simply becomes
\[
e^{-\frac{\hat{I}(\rho,\rho)}{2\hat{I}(\rho,\delta)}\delta}\prod_{n=1}^\infty \bigg( \big(1-e^{-n\delta}\big)^l \prod_{\alpha \in R_f}\big(1-e^{-\alpha-n\delta}\big) \bigg). \]
As for its right hand side, a direct computation shows
\begin{align*}
ut_\gamma(\rho)-\rho-\frac{\hat{I}(\rho,\rho)}{2\hat{I}(\rho,\delta)}\delta 
=
&u\big(\rho+\hat{I}(\rho,\delta)\gamma\big)-\rho-\frac{1}{2\hat{I}(\rho,\delta)}\cdot \hat{I}\big(\rho+\hat{I}(\rho,\delta)\gamma, \rho+\hat{I}(\rho,\delta)\gamma\big)\delta \\
=
&u\big(\rho_f+\hat{I}(\rho,\delta)\gamma\big)-\rho_f-\frac{1}{2\hat{I}(\rho,\delta)}\cdot \hat{I}\big(\rho+\hat{I}(\rho,\delta)\gamma, \rho+\hat{I}(\rho,\delta)\gamma\big)\delta, 
\end{align*}
 for $u \in W(R_f)$ and $\gamma \in Q(R_f^{\vee})$, by \eqref{Weyl-untwisted}, it becomes
 \[\sum_{\gamma \in Q(R_f^\vee)}
 \bigg(\frac{\sum_{u \in W(R_f)}\det_{F_f}(u)e^{u\big(\rho_f+\hat{I}(\rho,\delta)\gamma\big)-\rho_f}}{\prod_{\alpha \in R_f^+}\big(1-e^{-\alpha}\big)}\bigg)
 e^{-\frac{1}{2\hat{I}(\rho,\delta)}\cdot I_f\big(\rho+\hat{I}(\rho,\delta)\gamma, \rho+\hat{I}(\rho,\delta)\gamma\big)\delta}.
 \]
 Hence, we obtain
 \begin{equation}\label{Macdonald-step1-I}
 \begin{split}
 & e^{-\frac{\hat{I}(\rho,\rho)}{2\hat{I}(\rho,\delta)}\delta}\prod_{n=1}^\infty \bigg( \big(1-e^{-n\delta}\big)^l \prod_{\alpha \in R_f}\big(1-e^{-\alpha-n\delta}\big) \bigg) \\
 =
 &
 \sum_{\gamma \in \hat{I}(\rho,\delta)Q(R_f^\vee)}
 \bigg(\frac{\sum_{u \in W(R_f)}\det_{F_f}(u)
 e^{u\big(\rho_f+\gamma\big)}}{e^{\rho_f}\prod_{\alpha \in R_f^+}\big(1-e^{-\alpha}\big)}\bigg)
 e^{-\frac{1}{2\hat{I}(\rho,\delta)}\cdot I_f \big(\rho+\gamma, \rho+\gamma\big)\delta}.
 \end{split}
 \end{equation}

Second, we consider the specialization $e^{\alpha_i}\, \mapsto\, 1$ ($1\leq i\leq l$). For this purpose, let us calculate $\hat{I}(\rho,\delta)$ and $\hat{I}(\rho,\rho)$. By definition, we have
\begin{equation}\label{I-value-1}
 \hat{I}(\rho,\delta)=\hat{I}(\rho,\alpha_0)+\hat{I}(\rho,\theta)=\frac{\hat{I}(\theta,\theta)}{2}\bigg( 1+\sum_{i=1}^l a_i^\vee \bigg)=\frac{\hat{I}(\theta,\theta)}{2}\dCox{X_l^{(1)}},
\end{equation}
from which it follows from Freudenthal de Vries' strange formula (cf. \S 47.10 of \cite{Freudenthal1969}, \cite{FeganSteer1989} or Exercise 6.4 of \cite{MoodyPianzola1995}) that 
\[
\frac{\hat{I}(\rho,\rho)}{2\hat{I}(\rho,\delta)}=
\frac{2}{I_f(\theta,\theta)} \cdot \frac{I_f(\rho_f,\rho_f)}{2\dCox{X_l^{(1)}}}=
\frac{\dim \fg(R_f)}{24},
\]
where $\fg(R_f)$ is the simple Lie algebra of type $X_l$. Thus, the left hand side of \eqref{Macdonald-step1-I} becomes
\[
e^{-\frac{\dim \fg(R_f)}{24}\delta}\prod_{n=1}^\infty \big(1-e^{-n\delta} \big)^{l+\vert R_f\vert} \\
=
\Big( e^{-\frac{1}{24}\delta}\prod_{n=1}^\infty\big(1-e^{-n\delta} \big)\Big)^{\dim \fg(R_f)}=\eta\big(e^{-\delta}\big)^{\dim \fg(R_f)},
\]
where $\eta(q)$ is the Dedekind eta function: $\eta(q)=q^{\frac{1}{24}}\prod_{n=1}^\infty(1-q^n)$ (usually, viewed as a function on $\tau$ 
in the upper half plane, where $q=e^{2\pi i \tau}$).  Recalled that the \textbf{weight lattice} of $R_f$ is defined by 
\[ P(R_f)=\big\{ \, \lambda \in F_f\, \big\vert\, I_f(\lambda, \alpha_i^\vee) \in \Z\; \;1 \leq \,\forall\, i\leq l\, \big\}. \]
For the right hand side, since $Q(R_f^\vee)$ is a sublattice of $\dfrac{2}{I_f(\theta,\theta)}Q(R_f) \subset F_f$, it follows from \eqref{I-value-1} that $\hat{I}(\rho,\delta)Q(R_f^\vee) \subset \dCox{X_l^{(1)}} Q(R_f)$. Hence, it is sufficient to compute the specialization $e^{\alpha_i} \mapsto 1$ ($1\leq i\leq l$) of
\[ \frac{\sum_{u \in W(R_f)}\det_{F_f}(u)
e^{u\big(\rho_f+\lambda \big)}}{e^{\rho_f}\prod_{\alpha \in R_f^+}\big(1-e^{-\alpha}\big)}
\in \bigoplus_{\mu \in P(R_f)} \R e^\mu, \]
 for any $\lambda \in P(R_f)$. 
Regard $e^{\mu}$ ($\mu \in P(R_f)$) as a function on $F_f$ defined by $e^{\mu}(x)=e^{I_f(\mu,x)}$, and evaluate the above fraction at $t \rho_f$ ($t \in \R$). As 
\begin{align*}
& \Big(\sum_{u \in W(R_f)}\det{}_{F_f}(u)e^{u(\lambda+\rho_f)}\Big)(t\rho_f) \\
=
&\sum_{u \in W(R_f)}\det{}_{F_f}(u)e^{I_f\big(u(\lambda+\rho_f),t\rho_f\big)}=\sum_{u \in W(R_f)}\det{}_{F_f}(u^{-1})e^{I_f\big(u^{-1}(\rho_f), t(\lambda+\rho_f)\big)} \\
=
&\Big(\sum_{u \in W(R_f)}\det{}_{F_f}(u)e^{u(\rho_f)}\Big)\big(t(\lambda+\rho_f)\big), 
\end{align*}
it follows from \eqref{eq:denominator-id-fin} that
\begin{align*}
\Bigg(\frac{\sum_{u \in W(R_f)}\det_{F_f}(u)
e^{u\big(\rho_f+\lambda \big)}}{e^{\rho_f}\prod_{\alpha \in R_f^+}\big(1-e^{-\alpha}\big)}\Bigg)(t\rho_f)&=
\prod_{\alpha \in R_f^+}\bigg( 
\frac{e^{\frac{t}{2}I_f(\lambda+\rho_f,\alpha)}-e^{-\frac{t}{2}I_f(\lambda+\rho_f,\alpha)}}{e^{\frac{t}{2}I_f(\rho_f,\alpha)}-
e^{-\frac{t}{2}I_f(\rho_f,\alpha)}}\bigg) \; \\
&\overset{t \rightarrow 0}{\longrightarrow} \; 
\prod_{\alpha \in R_f^+} \frac{I_f(\lambda+\rho_f,\alpha)}{I(\rho_f,\alpha)}.
\end{align*}
Setting
\[ d_{X_l}(\lambda)=\prod_{\alpha \in R_f^+} \frac{I_f(\lambda+\rho_f,\alpha)}{I(\rho_f,\alpha)}
\]
for $\lambda \in P(R_f)$, we see that the right hand side of \eqref{Macdonald-step1-I} becomes
\[ \sum_{ \gamma \in \hat{I}(\rho,\delta)Q(R_f^\vee)}d_{X_l}(\gamma)e^{-\frac{I_f(\rho_f+\gamma,\rho_f+\gamma)}{2\dCox{X_l^{(1)}}}\cdot \frac{2}{I_f(\theta,\theta)}\delta}.
\]
Thus, we obtain the so-called Macdonald identity:
\begin{thm}[cf. Theorem 8.7 in \cite{Macdonald1972}]
Let $R_f$ be a finite root system of type $X_l$ and $R_{af}$ be an affine root system of type $X_l^{(1)}$. 
\begin{equation}\label{Macdonald-id-1}
\eta\big(e^{-\delta}\big)^{\dim \fg(R_f)}=
\sum_{ \gamma \in \hat{I}(\rho,\delta)Q(R_f^\vee)}d_{X_l}(\gamma)e^{-\frac{I_f(\rho_f+\gamma,\rho_f+\gamma)}{2\dCox{X_l^{(1)}}}\cdot \frac{2}{I_f(\theta,\theta)}\delta}.
\end{equation}
\end{thm}
Notice that $\dim \fg(R_f)=(\Cox{X_l}+1)\big\vert \Pi_f\big\vert$  by Corollary 6.8 and Theorem 8.4 of \cite{Kostant1959}.
\subsection{Macdonald identity II: twisted case}\label{sect:Macdonald2}
Here, we basically follow the same idea as in the last subsection to obtain the Macdonald identity for a twisted affine root system $R_{af}=R(X_l^{(t)})$
not of type $BC_l^{(2)}$. \\

First, we divide the both side of \eqref{eq:denominator-id-affine-II} by $e^{\rho}\prod_{\alpha \in R_f^+}(1-e^{-\alpha})$. Its left hand side simply becomes
\[
e^{-\frac{\hat{I}(\rho,\rho)}{2\hat{I}(\rho,\delta)}\delta}\prod_{n=1}^\infty \bigg( \big(1-e^{-n\delta}\big)^{\vert \Pi_{f,s} \vert}\big(1-e^{-tn\delta}\big)^{\vert \Pi_{f,l} \vert}
\prod_{\alpha \in (R_f)_s}\big(1-e^{-\alpha-n\delta}\big)\prod_{\alpha \in (R_f)_l}\big(1-e^{-\alpha-tn\delta}\big)  \bigg). \]
As for its right hand side, by a similar computation in the last subsection together with \eqref{Weyl-twisted}, it becomes
 \[\sum_{\gamma \in \frac{2}{\hat{I}(\theta_s,\theta_s)}Q(R_f)}
 \bigg(\frac{\sum_{u \in W(R_f)}\det_{F_f}(u)e^{u\big(\rho_f+\hat{I}(\rho,\delta)\gamma\big)-\rho_f}}{\prod_{\alpha \in R_f^+}\big(1-e^{-\alpha}\big)}\bigg)
 e^{-\frac{1}{2\hat{I}(\rho,\delta)}\cdot I_f\big(\rho+\hat{I}(\rho,\delta)\gamma, \rho+\hat{I}(\rho,\delta)\gamma\big)\delta}.
 \]
 Hence, we obtain
 \begin{equation}\label{Macdonald-step1-II}
 \begin{split}
 &
 e^{-\frac{\hat{I}(\rho,\rho)}{2\hat{I}(\rho,\delta)}\delta}\prod_{n=1}^\infty \bigg( \big(1-e^{-n\delta}\big)^{\vert \Pi_{f,s} \vert}\big(1-e^{-tn\delta}\big)^{\vert \Pi_{f,l} \vert}
\prod_{\alpha \in (R_f)_s}\big(1-e^{-\alpha-n\delta}\big)\prod_{\alpha \in (R_f)_l}\big(1-e^{-\alpha-tn\delta}\big)  \bigg) \\
 =
 &\sum_{\gamma \in 2\frac{\hat{I}(\rho,\delta)}{\hat{I}(\theta_s,\theta_s)}Q(R_f)}
 \bigg(\frac{\sum_{u \in W(R_f)}\det_{F_f} (u)
 e^{u\big(\rho_f+\gamma\big)}}{e^{\rho_f}\prod_{\alpha \in R_f^+}\big(1-e^{-\alpha}\big)}\bigg)
 e^{-\frac{1}{2\hat{I}(\rho,\delta)}\cdot I_f \big(\rho+\gamma, \rho+\gamma\big)\delta}.
 \end{split}
 \end{equation}
 
 Second, we consider the specialization $e^{\alpha_i}\, \mapsto\, 1$ ($1\leq i\leq l$). For this purpose, let us calculate $\hat{I}(\rho,\rho)$, and $\hat{I}(\rho,\delta)$. A similar computation to \eqref{I-value-1} yields
\begin{equation}\label{I-value-2}
 \hat{I}(\rho,\delta)=\frac{\hat{I}_f(\theta_s,\theta_s)}{2}\dCox{X_l^{(t)}},
\end{equation}
In particular, we have
\[ \frac{\hat{I}(\rho,\rho)}{2\hat{I}(\rho,\delta)}=\frac{2}{I_f(\theta_s,\theta_s)} \cdot \frac{I_f(\rho_f,\rho_f)}{2\dCox{X_l^{(t)}}}. \]
 Let us compute $\dfrac{I_f(\rho_f,\rho_f)}{2\dCox{X_l^{(t)}}}$. 
 For this purpose, we realize $R_f$ and $R_{af}$ as the image of a folding,
explained in \S \ref{sect:folding}.  Let $\widetilde{F}_f$ be a real vector space, containing $F_f$ as a subspace, equipped with positive definite metric $\tilde{I}_f$ whose restriction to $F_f \times F_f$ gives $I_f$. Let $\widetilde{R}_f$ be the finite root system belonging to $(\widetilde{F}_f,\tilde{I}_f)$ with a set  of simple roots $\widetilde{\Pi}_f$ whose Dynkin diagram admits a non-trivial automorphism $\tilde{\sigma}_f$ such that one has 
$\widetilde{F}_f^{H}=F_f$, $\Tr^{H}(\widetilde{R}_f)=R_f$ and $\Tr^{H}(\widetilde{\Pi}_f)=\Pi_f$, where we set $H=\langle \tilde{\sigma}_f \rangle$. 
In particular, it implies $\Tr^{H}(\widetilde{R}_f^+)=R_f^+$. Hence, in particular, we have the following:
\begin{enumerate}
\item
\[ (R_f)_s=\big\{\, \Tr^H(\alpha)\, \big\vert \, \alpha \in \big(\widetilde{R}_f\big)^{H}\, \big\}  \qquad \text{and} \qquad
   (R_f)_l=\big\{\, \Tr^H(\alpha)\, \big\vert \, \alpha \in \widetilde{R}_f \setminus \big(\widetilde{R}_f\big)^{H}\, \big\}.
\]
\item Set $\tilde{\rho}_f=\frac{1}{2}\sum_{\alpha \in \widetilde{R}_f^+} \alpha$. Then,
\[ \tilde{\rho}_f=\frac{1}{2}\bigg(\sum_{\alpha \in (\widetilde{R}_f^+)^H} \alpha+\sum_{\alpha \in \widetilde{R}_f^+ \setminus (\widetilde{R}_f^+)^H} \alpha\bigg)=\frac{1}{2}\Big( \sum_{\alpha \in (R_f^+)_s} \alpha+\sum_{\alpha \in (R_f^+)_l}\alpha\Big)=\rho_f. 
\]
\item Let $\tilde{\theta} \in \widetilde{R}_f^+$ be the highest weight. One has \; $\tilde{\theta}=\Tr^H(\tilde{\theta})=\theta_s$.
\end{enumerate}
Consider the real vector space $\widetilde{F}=\widetilde{F}_f\oplus \rad(I)$ equipped with a positive semidefinite metric $\tilde{I}$ such that $\tilde{I}\vert_{\widetilde{F}_f \times \tilde{F}_f}=\tilde{I}_f$ and $\tilde{I}\vert_{F \times F}= I$, where $F=F_f \oplus \rad(I)$ is clearly a subspace of $\widetilde{F}$. The automorphism $\tilde{\sigma}_f$ extends to the automorphism $\tilde{\sigma}_{af}$ of the affine root system $\widetilde{R}_{af}:=\widetilde{R}_f +\Z \delta$, fixing $\delta \in \rad(I)$.  
By definition, one has $\Tr^H(\widetilde{R}_{af})=R_{af}$. Let $Y_N$ be the type of the root system $\widetilde{R}_f$. Then, the type of the affine root system $\widetilde{R}_{af}$ is $Y_N^{(1)}$ and
\begin{lemma} $\dCox{X_l^{(t)}}=\dCox{Y_N^{(1)}}$. 
\end{lemma}
\begin{proof}
Let $\widetilde{\Pi}_{af}=\{ \beta_j\}_{0\leq j \leq N}$ be the set of simple roots of $\widetilde{R}_{af}$ containing $\widetilde{\Pi}_f$ and satisfying $\Tr^H(\widetilde{\Pi}_{af})=\Pi_{af}$.  Let $\{ (a_j)_{Y_N} \}_{0\leq j\leq N}$ and $\{ (a_j^\vee)_{Y_N}\}_{0\leq j \leq N}$ be the labels and colabels, respectively, of $\widetilde{\Pi}_{af}$. By definition, $\dCox{Y_N^{(1)}}=\sum_{j=0}^N (a_i^\vee)_{Y_N}$. Since $\widetilde{R}_f$ hence $\widetilde{R}_{af}$ is of simply-laced type, it follows that $(a_j^\vee)_{Y_N}=(a_j)_{Y_N}$ for any $0\leq j\leq N$, whereas for $R_f$, one has
\[
\theta_s^\vee=\frac{2}{I_f(\theta_s,\theta_s)}\theta_s=\frac{2}{I_f(\theta_s,\theta_s)}\Big(\sum_{i=1}^l a_i \alpha_i\Big)
=\sum_{ \alpha_i \in \Pi_{f,s}} a_i \alpha_i^\vee+\sum_{ \alpha_i \in \Pi_{f,l}}ta_i \alpha_i^\vee \]
which implies 
\[ a_i^\vee=\begin{cases} \; a_i \; & \alpha_i \in \Pi_{f,s}, \\ \; ta_i\; & \alpha_i \in \Pi_{f,l}. \end{cases} \]
As $\tilde{\theta}=\Tr^H(\widetilde{\theta})=\theta_s$, it follows that $(a_j)_{Y_N}=a_i$ if $\Tr^H(\beta_j)=\alpha_i$. Since, $t=\vert \langle \tilde{\sigma}_f \rangle \vert=\big\vert H.\beta \big\vert$ if $\beta \in \widetilde{R}_f \setminus \widetilde{R}_f^H$, we obtain
\begin{align*}
\dCox{X_l^{(t)}}=
&1+\sum_{i=1}^l a_i^\vee=1+\sum_{\substack{1\leq i\leq l \\ \alpha_i \in \Pi_{f,s}}}a_i+\sum_{\substack{1\leq i\leq l \\ \alpha_i \in \Pi_{f,l}}} ta_i 
\\
=
&1+\sum_{\substack{1\leq j\leq N \\ \beta_j \in \widetilde{\Pi}_{f}^H}}(a_j)_{Y_N}+\sum_{\substack{1\leq j\leq N \\ \beta_j \in \widetilde{\Pi}_f \setminus \widetilde{\Pi}_{f}^H}}(a_j)_{Y_N}=\dCox{Y_N^{(1)}}. 
\end{align*}
\end{proof}
By this lemma and Freudenthal de Vries' strange formula for $\widetilde{R}_f$, we obtain
\[ \frac{I_f(\rho_f,\rho_f)}{2\dCox{X_l^{(t)}}}=\frac{\tilde{I}_f(\tilde{\rho}_f,\tilde{\rho}_f)}{2\dCox{Y_N^{(1)}}}=\frac{\dim \fg(Y_N)}{24}\cdot \frac{\tilde{I}(\tilde{\theta},\tilde{\theta})}{2},
\]
which implies 
\[ \frac{\hat{I}(\rho,\rho)}{2\hat{I}(\rho,\delta)}=\frac{2}{I_f(\theta_s,\theta_s)} \cdot \frac{\dim \fg(Y_N)}{24} \cdot \frac{\tilde{I}_f(\tilde{\theta},\tilde{\theta})}{2}=\frac{\dim \fg(Y_N)}{24}.
\]
\vskip 0.1in
Now, we compute $\vert (R_f)_s \vert$ and $\vert (R_f)_l \vert$.  Let $c \in W(R_f)$ be a Coxeter element, and 
$\langle c\rangle$ the group generated by $c$. Consider the action of $\langle c\rangle$
on $R_f$. By Theorem 8.1 in \cite{Kostant1959}, each $\langle c \rangle$-orbit contains a positive root, say $\alpha$, which is sent to a negative root, i.e., $c(\alpha) \in -R_f^+$. Such a positive root can be expressed as follows. Suppose that $c=s_{\alpha_l}s_{\alpha_{l-1}}\cdots s_{\alpha_1}$. Then, there exists $1\leq j \leq l$ such that $\alpha=s_{\alpha_1}s_{\alpha_2} \cdots s_{\alpha_{j-1}}(\alpha_j)$
Moreover, by Corollay 8.2 \textit{ibid}, each $\langle c \rangle$-orbit contains exactly $\Cox{X_l}$ elements. 
Thus, we have
\begin{lemma} \label{lemma2}
$\vert (R_f)_s \vert=\Cox{X_l}\vert \Pi_{f,s} \vert$ and $ \vert (R_f)_l \vert=\Cox{X_l}\vert \Pi_{f,l} \vert$.
\end{lemma}
We now express $\dim \fg(Y_N)$ in terms of $\vert \Pi_{f,s} \vert$ and $\vert \Pi_{f,l} \vert$. 
Since $t=\vert \langle \tilde{\sigma}_f \rangle \vert=\big\vert H.\beta \big\vert$ if $\beta \in \widetilde{R}_f \setminus \widetilde{R}_f^H$, 
one has 
\begin{align*}
\vert \tilde{R}_f\vert =
&\big\vert (\widetilde{R}_f)^{\langle \tilde{\sigma}_f\rangle} \big\vert+
\big\vert \widetilde{R}_f \setminus \widetilde{R}_f^{\langle \tilde{\sigma}_f\rangle} \big\vert=\big\vert (R_f)_s \big\vert+t\big\vert (R_f)_l \big\vert, \\
N=
&\big\vert \widetilde{\Pi}_f\vert=\big\vert \widetilde{\Pi}_f^{\langle \tilde{\sigma}_f \rangle} \big\vert+
\big\vert \widetilde{\Pi}_f \setminus \widetilde{\Pi}_f^{\langle \tilde{\sigma}_f \rangle} \big\vert
=\big\vert \Pi_{f,s} \big\vert+t\big\vert \Pi_{f,l}\big\vert.
\end{align*}
Hence, Lemma \ref{lemma2} implies
\[ \dim \fg(Y_N)=N+\vert \tilde{R}_f \vert=(\Cox{X_l}+1)\big(\big\vert \Pi_{f,s} \big\vert+t\big\vert \Pi_{f,l}\big\vert). 
\]
Hence, the left hand side of \eqref{Macdonald-step1-II} becomes
\begin{align*}
&e^{-\frac{1}{24}(\Cox{X_l}+1)\big(\big\vert \Pi_{f,s} \big\vert+t\big\vert \Pi_{f,l}\big\vert)\delta}
\prod_{n=1}^\infty \Big( \big(1-e^{-n\delta}\big)^{(\Cox{X_l}+1)\big\vert \Pi_{f,s}\big\vert}\big(1-e^{-tn\delta}\big)^{(\Cox{X_l}+1)\big\vert \Pi_{f,l}\big\vert}\Big) \\
=
&
\Big(\eta\big(e^{-\delta}\big)^{\big\vert \Pi_{f,s}\big\vert}\eta\big(e^{-t\delta}\big)^{\big\vert \Pi_{f,l}\big\vert}\Big)^{\Cox{X_l}+1}.
\end{align*}
As for the right hand side of \eqref{Macdonald-step1-II},  \eqref{I-value-2} implies 
\[ 2\frac{\hat{I}(\rho,\delta)}{\hat{I}(\theta_s,\theta_s)}=\frac{2}{I_f(\theta_s,\theta_s)} \cdot \frac{I_f(\theta_s,\theta_s)}{2}\dCox{X_l^{(t)}}
=\dCox{X_l^{(t)}} \]
from which a similar argument to the last subsection implies that the right hand side of \eqref{Macdonald-step1-II} becomes
\[ \sum_{\gamma \in \dCox{X_l^{(t)}} Q(R_f)} d_{X_l}(\gamma)
e^{-\frac{I_f(\rho_f+\gamma,\rho_f+\gamma)}{2\dCox{X_l^{(t)}}}\cdot \frac{2}{I_f(\theta_s,\theta_s)}\delta}.
\]
Thus, we obtain the so-called Macdonald identity:
\begin{thm}[cf. Theorem 8.11 in \cite{Macdonald1972}]
Let $R_{af}$ be a twisted affine root system of type $X_l^{(t)}$ with $t\neq 1$ and $R_f$ be its finite quotient system. Assume that $R_{af}$ is not of type $BC_l^{(2)}$. 
\begin{equation}\label{Macdonald-id-2}
\Big(\eta\big(e^{-\delta}\big)^{\big\vert \Pi_{f,s}\big\vert}\eta\big(e^{-t\delta}\big)^{\big\vert \Pi_{f,l}\big\vert}\Big)^{\Cox{X_l}+1}=
\sum_{ \gamma \in \dCox{X_l^{(t)}} Q(R_f)}d_{X_l}(\gamma)e^{-\frac{I_f(\rho_f+\gamma,\rho_f+\gamma)}{2\dCox{X_l^{(t)}}}\cdot \frac{2}{I_f(\theta_s,\theta_s)}\delta}.
\end{equation}
\end{thm}

\subsection{Type $BC_l^{(2)}$}
Here we treat the case when $R_{af}$ is of type $BC_l^{(2)}$ rapidly.  As in the last subsections, we assume that $R_{af}$ belongs to $(F,I)\subset (\widehat{F},\hat{I})$. \\

Let $(R_{af})_s, (R_{af})_m$ and $(R_{af})_l$ be the subset of short, middle length and long roots of $R_{af}$, respectively. Set $(R_f')_\ast=R_f' \cap (R_{af})_\ast$ ($\ast \in \{m,l\}$). By definition, 
\begin{align*}
(R_{af})_s=
&\Big\{\,  \frac{1}{2}\big( \alpha+(2n-1)\delta \big)\, \Big\vert \, \alpha \in (R_f')_l, \; n \in \Z\, \Big\}, \\
(R_{af})_m=
&(R_f')_m+\Z \delta, \qquad (R_{af})_l=(R_f')_l+2\Z \delta.
\end{align*}
The denominator identity of type $BC_l^{(2)}$ gives
\begin{align*}
&e^{\rho-\frac{\hat{I}(\rho,\rho)}{2\hat{I}(\rho,\delta)}\delta}
\prod_{\alpha \in (R_f')^+}(1-e^{-\alpha}) 
\prod_{n=1}^\infty \Big( \big(1-e^{-n\delta}\big)^l \\
&\phantom{e^{\rho-\frac{\hat{I}(\rho,\rho)}{2\hat{I}(\rho,\delta)}}\prod_{\alpha \in (R_f')^+}(1-e^{-\alpha})}\times 
\prod_{\alpha \in (R_f')_m}\big( 1-e^{\alpha-n\delta}\big) 
\prod_{\alpha \in (R_f')_l}\big( 1-e^{\alpha-2n\delta}\big)\big(1-e^{\frac{1}{2}(\alpha-(2n-1)\delta)}\big)\Big) 
\\
=
&\sum_{w \in \widehat{W}(R_{af})} \det{}_{\widehat{F}}(w)e^{w(\rho)-\frac{\hat{I}(\rho,\rho)}{2\hat{I}(\rho,\delta)}\delta}.
\end{align*}
Dividing the both sides by $e^{\rho}\prod_{\alpha \in (R_f')^+}(1-e^{-\alpha})$ and specializing $e^{-\alpha_i}$ $(0\leq i\leq l-1)$ to $1$, 
we obtain the so-called the Macdonald identity:
\begin{thm}[cf. Type $BC_l$ (6)(b) of Appendix 1 in \cite{Macdonald1972}]
Let $R_{af}$ be an affine root system of type $BC_l^{(2)}$ and $R_f'$ be a finite root subsystem of type $C_l$. 
\begin{equation}\label{Macdonald-identity-3}
\bigg(  \eta\big(e^{-\frac{1}{2}\delta}\big)^{2}\eta\big(e^{-\delta}\big)^{2l-3}\eta\big(e^{-2\delta}\big)^{2}\bigg)^{l}
=\sum_{\gamma \in \hat{I}(\rho,\delta)Q((R_f')^\vee)}d_{C_l}(\gamma)e^{-\frac{1}{2\hat{I}(\rho,\delta)}\hat{I}(\rho+\gamma,\rho+\gamma)\delta}. 
\end{equation}
\end{thm}
Details are left to the reader.
\subsection{Postscript}
The article \cite{Macdonald1972} contains quite intriguing computations. 

For example, I. Macdonald has classified affine root systems even of non-reduced type and he obtained $4$ types:
$R(BCC_l)$, $R(C^\vee BC_l)$, $R(BB_l^\vee)$ and $R(C^\vee C_l)$, but its denominator identities were not given. 
In 1978, V. Kac \cite{Kac1978} considered a certain class infinite dimensional Lie superalgebras which contains $4$ types of affine Lie superalgebras: $B^{(1)}(0,l)$, $A^{(4)}(0,2l)$, $A^{(2)}(0,2l-1)$ and $C^{(2)}(l+1)$.  Indeed, up to $\Z/2\Z$-structure, we have the next correspondence:
\begin{center}
\begin{tabular}{|c||c|c|c|c|} \hline 
Non-reduced type & $BCC_l$ & $C^\vee BC_l$ & $BB_l^\vee$ & $C^\vee C_l$  \\ \hline \hline
Affine super          & $B^{(1)}(0,l)$ & $A^{(4)}(0,2l)$ & $A^{(2)}(0,2l-1)$ & $C^{(2)}(l+1)$ \\ \hline
\end{tabular}
\end{center}
where those in the same column are the same root system ! Moreover, their Macdonald identities coincides with those for types $BC_l^{(2)}$, $B_l^{(2)}$, $B_l^{(1)}$ and $B_l^{(2)}$, respectively. Notice that these $4$ affine root systems of reduced type are root subsystems of the corresponding non-reduced type consisting of the indivisible roots (cf. \cite{Macdonald1972}). There still seems to have 
many hidden structures to be explored in \cite{Macdonald1972}.

\bibliographystyle{plain}
\bibliography{Elliptic}
\phantom{ABC}
\vskip 0.5in

{\small
\begin{tabular}{@{}l@{}}%
\textsc{K. Iohara, }\\
\text{Univ Lyon, Universit\'{e} Claude Bernard Lyon 1,} \\ 
\text{CNRS, Institut Camille Jordan UMR 5208} \\
\text{F-69622 Villeurbanne, France} \\ 
\textit{E-mail address}: \texttt{iohara@math.univ-lyon1.fr}
\end{tabular}

\vskip 0.2in
\qquad \qquad

\begin{tabular}{@{}l@{}}%
\textsc{Y. Saito, }\\
\text{Department of Mathematics,} \\ 
\text{Rikkyo University,} \\
\text{Toshima-ku, Tokyo 171-8501, Japan} \\ 
\textit{E-mail address}: \texttt{yoshihisa@rikkyo.ac.jp}
\end{tabular}

}

\end{document}